\documentclass[a4paper,12pt]{elsarticle}
\usepackage{amsmath}
\usepackage[english]{babel}
\usepackage{graphicx}
\usepackage[usenames,dvipsnames]{color} 
\usepackage[numbers]{natbib}
\usepackage[footnotesize,bf,sf,center]{caption}
\usepackage{float}
\usepackage[dvips]{epsfig}
\usepackage[bottom=2.5cm,top=2.5cm,left=3cm,right=2cm]{geometry}
\usepackage{calc}
\usepackage{soul}
\usepackage{amsmath,amsfonts,amssymb}
\usepackage{subfigure}
\usepackage{setspace}
\usepackage{amssymb}
\usepackage{mathrsfs}
\usepackage{algorithm}
\usepackage{algpseudocode}
\usepackage{dsfont}
\usepackage{epstopdf}
\usepackage[utf8]{inputenc}
\usepackage{amsfonts,amsmath}
\usepackage{nicefrac,xfrac}
\usepackage{multirow}
\newcommand{\ve}[1]{\mbox{\boldmath $#1$}}
\newtheorem{theorem}{Theorem}[section]

\newtheorem{lemma}[theorem]{Lemma}

\newtheorem{remark}{Remark}[section]

\newdefinition{rmk}{Remark}
\newcommand{\proof} [1]{ \noindent {\bf Proof.} #1 \hfill\rule{0.5em}{1.2ex} \par\medskip}
\usepackage{graphicx, ulem}
\usepackage{comment}

\begin{document}
\begin{frontmatter}

  \renewcommand\arraystretch{1.0}

    \title{\textbf{Stress-divergence, Laplacian, and rotational forms of the incompressible Navier--Stokes equations with variable viscosity}}
      
    \author{
   {\bf B.~Cisternas}$^1$, \ 
   {\bf D.R.Q.~Pacheco}$^{2,3}$,
   {\bf F.~Galarce}$^1$
        \& \  
   {\bf E.~Castillo}$^4$ \\
   {\small ${}^{1}$ School of Civil Engineering, Pontificia Universidad Católica de Valparaíso, Valparaíso, Chile} \\
   {\small ${}^{2}$ Chair for Computational Analysis of Technical Systems, RWTH Aachen University, Germany}\\
   {\small ${}^{3}$ Chair of Methods for Model-based Development in Computational Engineering, RWTH}\\
    {\small ${}^{4}$ Department of Mechanical Engineering, University of Santiago de Chile, Santiago, Chile}}
\begin{abstract}
In the Navier--Stokes equations, incompressibility allows rewriting the viscous term in various forms leading to distinct numerical properties and flow descriptions. Furthermore, models accounting for non-Newtonian, thermal or turbulent effects often break the constant-viscosity assumption, thereby producing additional consistency terms. In this context, the present work compares the classical symmetric-gradient diffusion term with more recent variable-viscosity generalizations of the Laplacian and rotational forms. We discuss, analyze and test their differences with respect to implementation, efficiency, numerical stability and outflow boundary conditions. With a focus on time-dependent flows, we consider second-order implicit-explicit (IMEX) temporal discretizations aimed at improving efficiency and numerical stability. Through a rigorous stability analysis, we show how selected explicit treatments can bypass algorithmic nonlinearities without inducing CFL conditions. Our numerical results highlight important differences between the three viscous formulations---especially in the presence of outflow boundaries, for which the generalized Laplacian form proves more suitable in diffusion-dominated regimes. 
\end{abstract}
\end{frontmatter}

\section{Introduction}
The common assumption of constant viscosity is often violated in practical flow scenarios, such as in the presence of temperature gradients, non-Newtonian behavior, or turbulence closures. Allowing for variable viscosity in the Navier--Stokes equations renders many traditional numerical tools---such as any classical solver relying on the Laplacian form of the diffusive term---unusable. In this context, the literature on numerical methods for incompressible variable-viscosity flows has been steadily expanding. Recent developments include stabilization methods \cite{Barrenechea2023,Barrenechea2024,Galarce2025}, fractional-step schemes \cite{Espinoza2025,El-Amrani2025,ElAmrani2025,Schussnig2026}, efficient iterative algorithms \cite{Deteix2024} and mixed formulations \cite{Anaya2021,Anaya2023,Caucao2025}. 

For variable viscosities $\nu$ with spatial gradients $\nabla\nu \neq \ve{0}$, the viscous term $\nabla\cdot(2\nu\nabla^{\mathrm{s}}\ve{u})$ cannot be reduced to a simple Laplacian. Therefore, until recently, the full stress divergence was commonly regarded as the only consistent form of the diffusive term (expressions such as $\nabla\cdot[\nu\nabla\ve{u}]$ or $\nu\Delta\ve{u}$ have been used \cite{Pyo2007,Schiavazzi2017} but are simply \textsl{inconsistent} with variable $\nu$). Although other forms have recently emerged that correctly generalize the Laplacian \cite{Barrenechea2024,Guesmi2023} and rotational \cite{Anaya2023,Stiller2020} formulations, their numerical and modeling properties are still not fully understood or tested. 

\citet{Zang1991} was one of the first authors to provide a detailed numerical comparison of modified forms of the incompressible Navier--Stokes equations. That same year, \citet{Gresho1991} published his seminal article cataloging several formulations of the viscous and convective terms---considering both viscosity and density as constants. Nearly a decade later, \citet{Guermond2000} extended some of those results to variable density, still assuming constant viscosity. The differences explored in those works include the resulting matrix structure and sparsity pattern, as well as the outflow (Neumann) boundary conditions associated with each variant. As we show in this work, for variable viscosity there are even further differences, such as numerical stability.

In volume-integral methods such as finite elements and finite volumes, Neumann boundary conditions arise naturally from the divergence theorem. Therefore, depending on how viscous and convective terms are treated, different natural outflow boundary conditions arise. That would already be  a complex matter \textsl{if} outflow data were usually available, but, as remarked by \citet{Gresho1991}, \textit{``these BCs are not set by Nature; rather,
they are set by the need to truncate a domain for computational reasons''}. This issue is especially relevant in biomedical applications, where the computational domain typically represents only a truncated portion of the vasculature \cite{romor2026, MELLA2026108145}. Open boundaries can also trigger numerical instabilities, for instance in the presence of strong backflow. Thus, there are many works in the literature addressing outflow BCs and their numerical treatment \cite{Papanastasiou1992,HEYWOOD1996,Barth2007,Dimakopoulos2012,Braack2014,Bertoglio2014,Bertoglio2016,Bertoglio2018}, although only a few consider variable viscosity \cite{Barth2007,Dimakopoulos2012}. 

This article presents, to the best of our knowledge, the first unified and systematic comparison between the stress-divergence form of the viscous term and variable-viscosity generalizations of the Laplacian and rotational variants. We also derive a rotational weak formulation that does not introduce vorticity as an additional unknown, thereby differing from the three-field formulation of \citet{Anaya2023}. The comparison addresses three aspects relevant to time-dependent simulations: algorithmic structure, temporal stability, and natural boundary conditions. Our assessment combines implementation aspects, stability analysis, and numerical examples, providing a comprehensible overview that will hopefully aid CFD experts and computational rheologists in selecting the most appropriate variant for the desired applications. 

The rest of this article is organized as follows. Section~\ref{sec2} introduces the model problem and derives the stress-divergence, generalized Laplacian, and generalized rotational forms, with emphasis on their weak formulations and natural boundary conditions. Section~\ref{sec3} presents the finite element discretization and the IMEX time-integration schemes used for the three formulations. Section~\ref{sec4} provides the temporal stability analysis and compares the resulting estimates. Section~\ref{sec5} reports the numerical experiments, including temporal convergence tests and variable-viscosity flow problems with outflow boundaries. Finally, Section~\ref{sec6} summarizes the main findings and discusses the implications for selecting a viscous formulation in practical computations.

\section{Preliminaries}
\label{sec2}
\subsection{Model problem}
In a finite time interval $(0,T]$ and a domain $\Omega\subset\mathbb{R}^{d}$ with Lipschitz boundary $\Gamma$, we consider the incompressible Navier--Stokes equations:
\begin{align}
\partial_t\ve{u} + \ve{u}\cdot\nabla\ve{u} - \nabla\cdot(2\nu\nabla^{\mathrm{s}}\ve{u}) + \nabla p &= \ve{f} && \text{in} \ \ \Omega\times(0,T]=:Q\, ,\label{momentum}\\
\nabla\cdot\ve{u} &= 0  && \text{in} \ \ \Omega\times(0,T]\, , \label{incompressibility}
\end{align}
equipped with appropriate boundary and initial conditions. The unknowns are the velocity $\ve{u}$ and the pressure $p$, while $\ve{f}$ is a given body force. The viscosity field $\nu(\ve{x},t)$ is assumed to satisfy
\begin{equation}
   0 < \nu_{\mathrm{min}} \leq \nu(\ve{x},t) \leq  \nu_{\mathrm{max}} < \infty \ \ \text{in} \ \bar{Q}\, ,
   \label{nu_lim}
\end{equation}
where $(\nu_{\mathrm{min}},\nu_{\mathrm{max}})$ are known constants, such as the asymptotic bounds of a rheological model. The viscosity may be either a prescribed field $\nu(\ve{x},t)$ or determined by a rheological law $\nu = \eta(\dot{\gamma}(\ve{u}))$, with the shear rate defined as
\begin{align}
   \dot{\gamma}(\ve{u}) :=  \sqrt{2 \nabla^{\mathrm{s}}\ve{u} : \nabla^{\mathrm{s}}\ve{u}}\, .
\end{align}
The non-Newtonian numerical examples in this article assume the shear-thinning Carreau--Yasuda law $\eta:\mathbb{R}^{+}\rightarrow (\nu_{\mathrm{min}},\nu_{\mathrm{max}}]$
\begin{equation}
  \eta(\dot{\gamma}) = \nu_{\mathrm{min}}  + (\nu_{\mathrm{max}} - \nu_{\mathrm{min}})\left[1 + \left(\lambda\dot{\gamma}\right)^a\right]^{\frac{n-1}{a}},\label{CY}
\end{equation}
where $a > 0$, $n\in [0,1)$, $\lambda > 0$ and $\nu_{\mathrm{max}} > \nu_{\mathrm{min}} > 0$ are empirical constants; many other popular models, such as the power-law and Carreau rheologies, are included in \eqref{CY} (e.g., $a=2$ for Carreau).

\subsection{Viscous reformulations}
For constant viscosity, most CFD codes rewrite the viscous term $-\nabla\cdot(2\nu\nabla^{\mathrm{s}}\ve{u})$ in its equivalent Laplacian form $-\nu\Delta\ve{u}$. The main computational advantage is the resulting algebraic structure, since the vector Laplacian acts on each velocity component separately, i.e., $(\Delta\ve{u})_i = \Delta u_i$, with $u_i$ denoting the $i$th spatial component of $\ve{u}$. This leads to block-diagonal velocity-velocity matrices, in contrast to the coupling induced by the symmetric-gradient form. Another (less popular) variant is the rotational form $\nu\nabla\times(\nabla\times\ve{u})$, which is closely related to formulations arising in electromagnetism \cite{Boffi2026}. Apart from algorithmic differences, an important distinction between these three formulations is their natural boundary conditions on outflow boundaries, as will be discussed later on.

For variable viscosity, until recently the stress-divergence form was the only one used in practice, since the plain Laplacian and curl-curl forms are not valid if $\nabla\nu \neq \ve{0}$. When the viscosity field is Lipschitz-continuous in space, however, the Laplacian and rotational forms can be generalized by retaining the corresponding viscosity-gradient contributions. 
Using Eq.~\eqref{incompressibility} and the vector-Laplacian identity $\Delta = \nabla(\nabla\cdot) - \nabla\times(\nabla\times)$, \citet{Anaya2021} proposed the reformulation
\begin{align}
    -\nabla\cdot(2\nu\nabla^{\mathrm{s}}\ve{u}) &= -\nu\nabla\cdot(\nabla\ve{u} + \nabla^{\top}\ve{u}) - 2\nabla^{\mathrm{s}}\ve{u}\nabla\nu \nonumber\\
    &= -\nu[\Delta\ve{u} + \nabla(\nabla\cdot\ve{u})] - 2\nabla^{\mathrm{s}}\ve{u}\nabla\nu\nonumber\\
    &= \nu[\nabla\times(\nabla\times\ve{u}) - 2\nabla(\nabla\cdot\ve{u})] - 2\nabla^{\mathrm{s}}\ve{u}\nabla\nu\nonumber\\
    &= \nu\nabla\times(\nabla\times\ve{u}) - 2\nabla^{\mathrm{s}}\ve{u}\nabla\nu\, ,\label{rot0}
\end{align}
which reduces to the standard curl-curl variant if $\nu$ is constant. The Laplacian form can likewise be generalized as 
\begin{align}
    -\nabla\cdot(2\nu\nabla^{\mathrm{s}}\ve{u}) &= -\nabla\cdot(\nu\nabla\ve{u}) -\nabla\cdot( \nu\nabla^{\top}\ve{u})\nonumber \\
    &= -\nabla\cdot(\nu\nabla\ve{u}) -\nu\nabla( \nabla\cdot\ve{u}) - \nabla^{\top}\ve{u}\nabla\nu\nonumber\\
    &= -\nabla\cdot(\nu\nabla\ve{u}) - \nabla^{\top}\ve{u}\nabla\nu \, .\label{GL0}
\end{align}
Alternative generalizations are also possible, as shown by \citet{Stiller2020}.

\subsection{Weak forms}
Although much of the mathematical literature focuses on homogeneous Dirichlet boundary conditions, the present study also considers problems with outflow boundaries. In particular, we seek weak formulations of the generalized Laplacian and rotational forms whose natural boundary conditions match those of their constant-viscosity counterparts, which requires constructing appropriate variational formulations. 

Let us consider a non-overlapping partition of the boundary $\partial\Omega$ into $\Gamma_N$ and $\Gamma_D$. On the latter, Dirichlet data $\ve{u}_D$ are prescribed for the velocity $\ve{u}$, while the type of data prescribed on $\Gamma_N$ depends on the specific form used for the viscous term. We shall for now consider sufficiently regular test functions $\ve{w}$ such that $\ve{w}|_{\Gamma_D} = \ve{0}$. For the standard stress-divergence formulation, integration by parts yields
\begin{align}
    -\int_{\Omega} \ve{w}\cdot\left[\nabla\cdot(2\nu\nabla^{\mathrm{s}}\ve{u})\right]\mathrm{d}\Omega = \int_{\Omega} 2\nu\nabla^{\mathrm{s}}\ve{u}:\nabla^{\mathrm{s}}\ve{w} - \int_{\Gamma_N}[(2\nu\nabla^{\mathrm{s}}\ve{u})\ve{n}]\cdot\ve{w}\,\mathrm{d}\Gamma\, ,\label{weakSD}
\end{align}
where $\ve{n}$ is the unit outward normal vector on $\Gamma_N$. For the generalized Laplacian (GL) variant, we write 
\begin{align}
    &-\int_{\Omega} \ve{w}\cdot\left[\nabla\cdot(\nu\nabla\ve{u}) + \nabla^{\top}\ve{u}\nabla\nu\right]\mathrm{d}\Omega\nonumber\\ 
    &= \int_{\Omega} \nu\nabla\ve{u}:\nabla\ve{w} - \big(\nabla^{\top}\ve{u}\nabla\nu\big)\cdot\ve{w}\, \mathrm{d}\Omega - \int_{\Gamma_N}[(\nu\nabla\ve{u})\ve{n}]\cdot\ve{w}\,\mathrm{d}\Gamma\, .\label{weakGL}
\end{align}
Concerning the generalized rotational form, so far the literature does not seem to contain weak forms based on integration by parts (\citet{Anaya2021,Anaya2023} introduced instead the vorticity $\ve{\omega} = \nabla\times\ve{u}$ as an additional unknown to circumvent second-order derivatives). Therefore, we construct a new weak form by first integrating the curl-curl term by parts: 
\begin{align}
    \int_{\Omega} \ve{w}\cdot\left[\nu\nabla\times(\nabla\times\ve{u}) - 2\nabla^{\mathrm{s}}\ve{u}\nabla\nu\right]\mathrm{d}\Omega &= \int_{\Omega} (\nabla\times\ve{u})\cdot\nabla\times(\nu\ve{w}) - \big[(\nabla^{\top}\ve{u} + \nabla\ve{u})\nabla\nu\big]\cdot\ve{w}\, \mathrm{d}\Omega\nonumber\\
    &+ \int_{\Gamma_N}\left[\ve{n}\times(\nu\nabla\times\ve{u})\right]\cdot\ve{w}\,\mathrm{d}\Gamma\, .\label{rotDerivation1}
\end{align}
We have that 
\begin{align}
    \int_{\Omega} (\nabla\times\ve{u})\cdot\nabla\times(\nu\ve{w}) \, \mathrm{d}\Omega &= \int_{\Omega} (\nabla\times\ve{u})\cdot(\nu\nabla\times\ve{w} + \nabla\nu\times\ve{w}) \, \mathrm{d}\Omega \nonumber\\
    &=\int_{\Omega} (\nabla\times\ve{u})\cdot(\nu\nabla\times\ve{w})  \, \mathrm{d}\Omega + \int_{\Omega} \left[\big(\nabla\ve{u} - \nabla^{\top}\ve{u}\big) \nabla\nu\right]\cdot\ve{w} \, \mathrm{d}\Omega,\label{rotDerivation2}
\end{align}
thanks to the identity $(\nabla\times\ve{b})\times\ve{a} = (\nabla\ve{b}-\nabla^{\top}\ve{b})\ve{a}$. The terms in Eqs.~\eqref{rotDerivation1} and \eqref{rotDerivation2} containing the product $\nabla\ve{u}\nabla\nu$ cancel each other out, so we are left with the \textsl{generalized rotational} (ROT) term
\begin{align}
\int_{\Omega} (\nabla\times\ve{u})\cdot(\nu\nabla\times\ve{w})  \, \mathrm{d}\Omega -\int_{\Omega}(2\nabla^{\top}\ve{u}\nabla\nu)\cdot\ve{w} -
\int_{\Gamma_N}\left[(\nu\nabla\times\ve{u})\times\ve{n}\right]\cdot\ve{w}\,\mathrm{d}\Gamma\, .\label{weakRot}
\end{align}
To the best of our knowledge, this specific weak formulation has not been reported previously.

Comparing the three weak forms \eqref{weakSD}, \eqref{weakGL}, \eqref{weakRot} reveals that the stress-divergence formulation is the only symmetric one, while the other two formulations contain the term $\nabla^{\top}\ve{u}\nabla\nu$ (which vanishes for constant $\nu$). The boundary data required on $\Gamma_N$ also differ for each formulation: after also integrating the pressure term by parts, the following natural boundary conditions are obtained:
\begin{align*}
    \text{Stress-divergence:}\ \ \ve{t} &= (2\nu\nabla^{\mathrm{s}}\ve{u})\ve{n} - p\ve{n}, && \text{(normal traction)} \\  
     \text{Laplacian:}\ \ \ve{t} &= (\nu\nabla\ve{u})\ve{n} - p\ve{n}, && \text{(normal pseudo-traction)} \\
       \text{Rotational:}\ \ \ve{t} &= (\nu\nabla\times\ve{u})\times\ve{n} - p\ve{n}\, . && 
\end{align*}
Depending on the envisioned application and available data, one or the other type of boundary data may be more advantageous. For example, in fluid-structure interaction it is often necessary to use the \textsl{real} normal traction $(2\nu\nabla^{\mathrm{s}}\ve{u})\ve{n} - p\ve{n}$ for the interface coupling; for flow problems with truncated outlets, which are our focus, pseudo-tractions are usually more suitable, as shown in the numerical examples and already well known for constant viscosity \cite{HEYWOOD1996}. In this work, we limit our discussion to the common ``do-nothing'' setting ($\ve{t}=\ve{0}$ on $\Gamma_N$), which produces a different outflow profile for each formulation.

\section{Discretization} \label{sec3}
For the spatial discretization we use Taylor--Hood elements: second- and first-order interpolation for velocity and pressure, respectively. For the time discretization, we use IMEX schemes that advance the solution without nonlinear iterations by treating both the convective velocity and the viscosity law explicitly. Moreover, the terms proportional to $\nabla^{\top}\ve{u}\nabla\nu$ in the generalized Laplacian and rotational formulations will be treated implicitly because they are generally non-positive and therefore may damage coercivity and stability of the overall bilinear forms. The time derivative is discretized backwardly:
\begin{align*}
    \partial_t\ve{u}|_{t_{n+1}} \approx \frac{3\ve{u}_{n+1} -4\ve{u}_{n} + \ve{u}_{n-1}}{2\tau},
\end{align*}
where $\tau = T/N$ is the (constant) time-step size, $N\in\mathbb{N}$. To match the second-order consistency of this finite difference, we consider the extrapolations
\begin{align}
   \ve{u}^{\star}_{n+1} &:= 2\ve{u}_{n} - \ve{u}_{n-1} 
\end{align}
and
\begin{align}
		\nu^{\star}_{n+1} := 	\begin{cases}
\nu(\ve{x},t_{n+1}),& \text{if $\nu(\ve{x},t)$ is known a priori (data),} 
\\
\eta(\dot{\gamma}(\ve{u}^{\star}_{n+1}))\, ,& \text{if $\nu$ is given through a rheological law $\eta$.}
\end{cases}&& 
\end{align}
In what follows, the $L^2(\Omega)$ product of two functions---be they scalars, vectors or tensors---will be written as $\langle \cdot,\cdot \rangle$. The velocity and pressure finite element spaces will be denoted as $X_h$ and $Y_h$, respectively.

In the stress-divergence formulation, each time step consists in finding $(\ve{u}_{n+1},p_{n+1})\in X_h\times Y_h$, with $\ve{u}_{n+1}|_{\Gamma_D} = \ve{u}_D|_{t_{n+1}}$, such that
\begin{align}
    &\left\langle\frac{3}{2\tau}\ve{u}_{n+1} + \ve{u}_{n+1}^{\star}\cdot\nabla\ve{u}_{n+1},\ve{w}\right\rangle  +\left\langle 2\nu_{n+1}^{\star}\nabla^{\mathrm{s}}\ve{u}_{n+1},\nabla^{\mathrm{s}}\ve{w}\right\rangle - \left\langle p_{n+1},\nabla\cdot\ve{w}\right\rangle + \left\langle \nabla\cdot\ve{u}_{n+1},q\right\rangle\nonumber\\
    &= \left\langle \frac{4\ve{u}_{n} - \ve{u}_{n-1}}{2\tau}+ \ve{f}_{n+1},\ve{w}\right\rangle \label{BDF2SD}
\end{align}
for all $(\ve{w},q)\in X_h\times Y_h$ fulfilling $\ve{v}|_{\Gamma_D} = \ve{0}$. The generalized Laplacian problem looks for $(\ve{u}_{n+1},p_{n+1})\in X_h\times Y_h$, with $\ve{u}_{n+1}|_{\Gamma_D} = \ve{u}_D|_{t_{n+1}}$, such that
\begin{align}
    &\left\langle\frac{3}{2\tau}\ve{u}_{n+1} + \ve{u}_{n+1}^{\star}\cdot\nabla\ve{u}_{n+1},\ve{w}\right\rangle  +\left\langle \nu_{n+1}^{\star}\nabla\ve{u}_{n+1},\nabla\ve{w}\right\rangle - \left\langle p_{n+1},\nabla\cdot\ve{w}\right\rangle + \left\langle \nabla\cdot\ve{u}_{n+1},q\right\rangle\nonumber \\
    &= \left\langle \frac{4\ve{u}_{n} - \ve{u}_{n-1}}{2\tau}+ \ve{f}_{n+1} + \underbrace{2\nabla^{\top}\ve{u}_{n}\nabla\nu_{n} - \nabla^{\top}\ve{u}_{n-1}\nabla\nu_{n-1}}_{\mathcal{O}(\tau^2)\text{-extrapolation of} \ \nabla^{\top}\boldsymbol{u}_{n+1}\nabla\nu_{n+1}},\ve{w}\right\rangle    \label{BDF2GL}
\end{align}
for all $(\ve{w},q)\in X_h\times Y_h$ fulfilling $\ve{v}|_{\Gamma_D} = \ve{0}$.

Pure rotational formulations are generally unstable for Lagrangian elements and require instead Nédélec-type spaces \cite{Boffi2013}. To circumvent this, we follow \citet{Anaya2021} by adding a \textsl{consistent} grad-div stabilization term scaled by a positive constant $\gamma$. We then seek
$(\ve{u}_{n+1},p_{n+1})\in X_h\times Y_h$, with $\ve{u}_{n+1}|_{\Gamma_D} = \ve{u}_D|_{t_{n+1}}$, such that
\begin{align}
    &\left\langle\frac{3}{2\tau}\ve{u}_{n+1} + \ve{u}_{n+1}^{\star}\cdot\nabla\ve{u}_{n+1},\ve{w}\right\rangle  +\left\langle \nu_{n+1}^{\star}\nabla\times\ve{u}_{n+1},\nabla\times\ve{w}\right\rangle + \gamma\left\langle \nu_{n+1}^{\star}\nabla\cdot\ve{u}_{n+1},\nabla\cdot\ve{w}\right\rangle + \nonumber\\
    &\left\langle \nabla\cdot\ve{u}_{n+1},q\right\rangle - \left\langle p_{n+1},\nabla\cdot\ve{w}\right\rangle  
    = \left\langle \frac{4\ve{u}_{n} - \ve{u}_{n-1}}{2\tau}+ \ve{f}_{n+1} + 2(2\nabla^{\top}\ve{u}_{n}\nabla\nu_{n} - \nabla^{\top}\ve{u}_{n-1}\nabla\nu_{n-1}),\ve{w}\right\rangle  \label{BDF2ROT}
\end{align}
for all $(\ve{w},q)\in X_h\times Y_h$ fulfilling $\ve{v}|_{\Gamma_D} = \ve{0}$. Well-posedness of this formulation can be proved by combining the identity $\|\nabla\times\ve{v}\|^2 + \|\nabla\cdot\ve{v}\|^2 = \|\nabla\ve{v}\|^2$ for $\ve{v}\in\ve{H}^1_0(\Omega)$ with the techniques we presented in a recent work \cite{Espinoza2025}.

With the tailored IMEX treatments proposed above for the viscous terms, all three methods lead to linear algebraic problems with symmetric diffusion matrices; the convective contribution remains nonsymmetric in general. The main algorithmic difference is that the generalized Laplacian formulation produces a block-diagonal velocity-velocity matrix with $d$ (2 our 3) identical blocks, which reduces assembly costs and memory requirements \cite{John2016}. This structure is especially advantageous for pressure-segregation schemes, since the resulting algorithm reduces to scalar subproblems (one for each velocity component and one for the pressure \cite{Barrenechea2024,Espinoza2025}).

\section{Temporal stability}\label{sec4}
\subsection{Useful results}
We denote by $\| \cdot \|$ and $\| \cdot \|_{\infty}$ the $L^2(\Omega)$ and $L^{\infty}(\bar{Q})$ norms, respectively, without distinguishing between scalar-, vector- or tensor-valued arguments. For concision, the stability analysis is carried out considering $\ve{f}=\ve{0}$ and $\ve{u}_{n}\in \ve{H}^1_0(\Omega)$ for all $t_n$. More general boundary conditions and forcing terms can be handled by using ideas from our previous work \cite{Barrenechea2024}. The analyses uses the skew-symmetry of the convective term:
\begin{align}
    \left\langle \ve{v}\cdot\nabla\ve{w},\ve{w}\right\rangle = 0 \label{skew} 
\end{align}
for any sufficiently regular $(\ve{w},\ve{v})$ such that $\ve{v}\cdot\ve{n}=0$ on $\partial\Omega$ and $\nabla\cdot\ve{v} = 0$. Also useful is the summation result (for $N\geq 2$)
\begin{align}
    \sum_{n=1}^{N-1}(a_{n+1} - \beta a_{n} - \gamma a_{n-1}) = (\beta + \gamma)(a_N - a_1) + \gamma(a_{N-1}-a_0) + (1-\beta-\gamma)\sum_{n=2}^{N} a_{n}\, ,\label{equality}
\end{align}
along with the identity
\begin{align}
     \langle 3\ve{v}_{n+1}-4\ve{v}_{n}+\ve{v}_{n-1},2\ve{v}_{n+1}\rangle = \|\ve{v}_{n+1}\|^2 - \|\ve{v}_{n}\|^2 + \|\ve{v}_{n+2}^{\star}\|^2 - \|\ve{v}_{n+1}^{\star}\|^2 +\|\delta^2\ve{v}_{n+1}\|^2  ,
    \label{identityBDF2}
\end{align}
where $\ve{v}_{n+1}^{\star} = 2\ve{v}_{n}-\ve{v}_{n-1}$ and
\begin{align*}
\delta^2\ve{v}_{n+1} :=&\ \ve{v}_{n+1}-2\ve{v}_{n}+\ve{v}_{n-1} \\
 =& \ \ve{v}_{n+1} - \ve{v}_{n+1}^{\star} \, .
\end{align*}
The stability analysis will also rely on the following Grönwall lemma \cite{Heywood1990}.
\begin{lemma}[Discrete Grönwall inequality]\label{Lem:Gronwall-unconditional}
Let $N\in\mathbb{N}$, and $\alpha,C,a_{n},b_{n}$ be non-negative numbers for $n=1,\ldots,N$. If these numbers satisfy
\begin{align}\label{N-1}
    a_{N} + \sum_{n=1}^{N}b_n \leq C +  \alpha\sum_{n=1}^{N-1}a_n  \, ,
\end{align}
then there holds
\begin{align}
    a_{N} + \sum_{n=1}^{N}b_n \leq C\mathrm{e}^{\alpha N} \ \ \text{for} \ \ N\geq 1 \, .
\end{align}
\end{lemma}

\begin{remark}
    BDF2 schemes are not self-starting because the first time step would require a ``pre-initial'' condition $\ve{u}_{-1}$. Therefore, the first step is usually computed using a corresponding BDF1 method. To avoid unnecessary technicalities, we assume that $\ve{u}_1$ is known and bounded, so that BDF2 can be used from the second step onward. 
\end{remark}

\subsection{Stress-divergence formulation}
The stability of the stress-divergence method is a classical result, but we state it here to enable a comparison with the other two methods.
\begin{theorem}[Stability of the linearized stress-divergence scheme]
For any time-step size $\tau=T/N$, $N\ge 2$, scheme \eqref{BDF2SD} satisfies the stability estimate
    \begin{equation}
    \begin{split}      
    &\|\ve{u}^{\star}_{N+1}\|^2 + \|\ve{u}_{N}\|^2   + \sum_{n=2}^{N}\left(8\tau\big\|\sqrt{\nu_{n}^{\star}}\,\nabla^{\mathrm{s}}\ve{u}_{n}\big\|^2 + \|\delta^2\ve{u}_{n}\|^2\right) \leq \|\ve{u}^{\star}_{2}\|^2 + \|\ve{u}_{1}\|^2\, .
     \label{stabilitySD}
     \end{split}
    \end{equation}
\end{theorem}
\proof{Setting $\ve{v}=4\tau\ve{u}_{n+1}$ in the weak form \eqref{BDF2SD} and using identities \eqref{skew} and \eqref{identityBDF2} yields
\begin{align}
(\|\ve{u}_{n+1}\|^2 + \|\ve{u}_{n+2}^{\star}\|^2) - (\|\ve{u}_{n}\|^2  + \|\ve{u}_{n+1}^{\star}\|^2) +\|\delta^2\ve{u}_{n+1}\|^2   + 8\tau\left\|\sqrt{\nu_{n+1}^{\star}}\,\nabla^{\mathrm{s}}\ve{u}_{n+1}\right\|^2 = 0 \,.\label{a}
\end{align}
Adding up from $n=1$ to $n=N-1$ proves the energy-decay property, since the bracketed terms in \eqref{a} telescope in time.}

\subsection{Generalized Laplacian formulation}
\begin{theorem}[Stability of the generalized Laplacian IMEX scheme]
For any time-step size $\tau=T/N$, $N\ge 2$, the IMEX scheme \eqref{BDF2GL} satisfies the stability estimate
    \begin{equation}
    \begin{split}      
    &\big\|\ve{u}_{N}\big\|^2  + \big\|\ve{u}_{N+1}^{\star}\big\|^2 + 8\varepsilon\tau\nu_{\mathrm{min}}\left(5\big\|\sqrt{\nu^{\star}_N}\,\nabla\ve{u}_{N}\big\|^2 + \big\|\sqrt{\nu^{\star}_{N-1}}\,\nabla\ve{u}_{N-1}\big\|^2\right)  \\
    &+ 4(1-10\varepsilon)\tau\sum_{n=2}^{N}\big\|\sqrt{\nu^{\star}_{n}}\,\nabla\ve{u}_{n}\big\|^2  \leq C\, \mathrm{exp}\left(\frac{\|\nabla\nu \|_{\infty}^2 T}{\varepsilon\nu_{\mathrm{min}}}\right)
     \, ,\label{stabilityGL}
     \end{split}
    \end{equation}
for any $\varepsilon \in \left(0,\frac{1}{10}\right]$, where $C = \|\ve{u}_{2}^{\star}\|^2 + \|\ve{u}_{1}\|^2  + 40\tau\varepsilon\big\|\sqrt{\nu^{\star}_{1}}\,\nabla\ve{u}_{1}\big\|^2 + 8\varepsilon\tau \big\|\sqrt{\nu^{\star}_{0}}\,\nabla\ve{u}_{0}\big\|^2$.
\end{theorem}
\proof{We begin by setting $\ve{v}=4\tau\ve{u}_{n+1}$ in \eqref{BDF2GL}, which yields
\begin{align}
 &\left(\|\ve{u}_{n+1}\|^2 + \|\ve{u}_{n+2}^{\star}\|^2\right) - \left(\|\ve{u}_{n}\|^2  + \|\ve{u}_{n+1}^{\star}\|^2\right) +\|\delta^2\ve{u}_{n+1}\|^2   + 4\tau\big\|\sqrt{\nu^{\star}_{n+1}}\,\nabla\ve{u}_{n+1}\big\|^2 \nonumber\\                     
 &= 4\tau\left\langle 2\nabla^{\top}\ve{u}_{n}\nabla\nu_{n} -\nabla^{\top}\ve{u}_{n-1}\nabla\nu_{n-1},\ve{u}_{n+1}\right\rangle\nonumber\\
 &= 4\tau\left\langle 2\nabla^{\top}\ve{u}_{n}\nabla\nu_{n} -\nabla^{\top}\ve{u}_{n-1}\nabla\nu_{n-1},\ve{u}^{\star}_{n+1} + \delta^2\ve{u}_{n+1}\right\rangle\nonumber\\
 &\leq 4\tau\|\nabla\nu\|_{\infty}\left(2\|\nabla^{\top}\ve{u}_{n}\| + \|\nabla^{\top}\ve{u}_{n-1}\|\right)\left(\|\ve{u}^{\star}_{n+1}\| + \|\delta^2\ve{u}_{n+1}\|\right)\nonumber\\
 &\leq 4\tau\frac{\|\nabla\nu\|_{\infty}}{\sqrt{\nu_{\mathrm{min}}}}\left(2\left\|\sqrt{\nu^{\star}_{n}}\,\nabla\ve{u}_{n}\right\| + \left\|\sqrt{\nu^{\star}_{n-1}}\,\nabla\ve{u}_{n-1}\right\|\right)\left(\left\|\ve{u}^{\star}_{n+1}\right\| + \left\|\delta^2\ve{u}_{n+1}\right\|\right),\label{estGL1}
\end{align}
where we have used Hölder's inequality. Now, Young's inequality gives us
\begin{align}
    4\tau\frac{\|\nabla\nu\|_{\infty}}{\sqrt{\nu_{\mathrm{min}}}}2\left\|\sqrt{\nu^{\star}_{n}}\,\nabla\ve{u}_{n}\right\|\left\|\delta^2\ve{u}_{n+1}\right\| &\leq \frac{\left\|\delta^2\ve{u}_{n+1}\right\|^2}{2} + 32\tau^2\frac{\|\nabla\nu\|_{\infty}^2}{\nu_{\mathrm{min}}}\|\sqrt{\nu^{\star}_{n}}\,\nabla\ve{u}_{n}\|^2 ,\\
    4\tau\frac{\|\nabla\nu\|_{\infty}}{\sqrt{\nu_{\mathrm{min}}}}\|\sqrt{\nu^{\star}_{n-1}}\,\nabla\ve{u}_{n-1}\|\left\|\delta^2\ve{u}_{n+1}\right\| &\leq \frac{\left\|\delta^2\ve{u}_{n+1}\right\|^2}{2} + 8\tau^2\frac{\|\nabla\nu\|_{\infty}^2}{\nu_{\mathrm{min}}}\|\sqrt{\nu^{\star}_{n-1}}\,\nabla\ve{u}_{n-1}\|^2 ,\\
    4\tau\frac{\|\nabla\nu\|_{\infty}}{\sqrt{\nu_{\mathrm{min}}}}2\left\|\sqrt{\nu^{\star}_{n}}\,\nabla\ve{u}_{n}\right\|\left\|\ve{u}_{n+1}^{\star}\right\| &\leq \frac{\tau\|\nabla\nu\|_{\infty}^2}{2\varepsilon\nu_{\mathrm{min}}}\left\|\ve{u}_{n+1}^{\star}\right\|^2 + 32\varepsilon\tau\|\sqrt{\nu^{\star}_{n}}\,\nabla\ve{u}_{n}\|^2 ,\\
    4\tau\frac{\|\nabla\nu\|_{\infty}}{\sqrt{\nu_{\mathrm{min}}}}\left\|\sqrt{\nu^{\star}_{n-1}}\,\nabla\ve{u}_{n-1}\right\|\left\|\ve{u}_{n+1}^{\star}\right\| &\leq \frac{\tau\|\nabla\nu\|_{\infty}^2}{2\varepsilon\nu_{\mathrm{min}}}\left\|\ve{u}_{n+1}^{\star}\right\|^2 + 8\varepsilon\tau\|\sqrt{\nu^{\star}_{n-1}}\,\nabla\ve{u}_{n-1}\|^2 , \label{estGL2}
\end{align}
for an arbitrary $\varepsilon > 0$. Combining estimates \eqref{estGL1}--\eqref{estGL2} yields
\begin{align*}
    &\left(\|\ve{u}_{n+1}\|^2 + \|\ve{u}_{n+2}^{\star}\|^2\right) - \left(\|\ve{u}_{n}\|^2  + \|\ve{u}_{n+1}^{\star}\|^2\right)   + 4\tau\big\|\sqrt{\nu^{\star}_{n+1}}\,\nabla\ve{u}_{n+1}\big\|^2
    \leq  \\ 
    &\frac{\tau\|\nabla\nu\|_{\infty}^2}{\nu_{\mathrm{min}}}\left(32\tau\|\sqrt{\nu^{\star}_{n}}\,\nabla\ve{u}_{n}\|^2 + 8\tau\|\sqrt{\nu^{\star}_{n-1}}\,\nabla\ve{u}_{n-1}\|^2 + \frac{1}{\varepsilon}\left\|\ve{u}_{n+1}^{\star}\right\|^2\right) \\ 
    &+ 8\varepsilon\tau(4\|\sqrt{\nu^{\star}_{n}}\,\nabla\ve{u}_{n}\|^2 + \|\sqrt{\nu^{\star}_{n-1}}\,\nabla\ve{u}_{n-1}\|^2)\, ,
\end{align*}
which we sum up from $n=1$ to $n=N-1$:
\begin{align*}
    &\|\ve{u}_{N}\|^2 + \|\ve{u}_{N+1}^{\star}\|^2 + 4\tau\sum_{n=1}^{N-1}\left(\big\|\sqrt{\nu^{\star}_{n+1}}\,\nabla\ve{u}_{n+1}\big\|^2 - 8\varepsilon\|\sqrt{\nu^{\star}_{n}}\,\nabla\ve{u}_{n}\|^2 - 2\varepsilon\|\sqrt{\nu^{\star}_{n-1}}\,\nabla\ve{u}_{n-1}\|^2\right)
      \\ &\leq \|\ve{u}_{1}\|^2  + \|\ve{u}_{2}^{\star}\|^2 + 
    \frac{\tau\|\nabla\nu\|_{\infty}^2}{\nu_{\mathrm{min}}}\sum_{n=1}^{N-1}\left(32\tau\|\sqrt{\nu^{\star}_{n}}\,\nabla\ve{u}_{n}\|^2 + 8\tau\|\sqrt{\nu^{\star}_{n-1}}\,\nabla\ve{u}_{n-1}\|^2 + \frac{1}{\varepsilon}\left\|\ve{u}_{n+1}^{\star}\right\|^2\right) .
\end{align*}
We rewrite the summation on the left-hand side using \eqref{equality}:
\begin{align*}
    &\sum_{n=1}^{N-1}\left(\big\|\sqrt{\nu^{\star}_{n+1}}\,\nabla\ve{u}_{n+1}\big\|^2 - 8\varepsilon\big\|\sqrt{\nu^{\star}_{n}}\,\nabla\ve{u}_{n}\big\|^2 - 2\varepsilon\big\|\sqrt{\nu^{\star}_{n-1}}\,\nabla\ve{u}_{n-1}\big\|^2\right)= \\
&10\varepsilon\left(\big\|\sqrt{\nu^{\star}_{N}}\,\nabla\ve{u}_{N}\big\|^2 - \big\|\sqrt{\nu^{\star}_{1}}\,\nabla\ve{u}_{1}\big\|^2\right) + 2\varepsilon\left(\big\|\sqrt{\nu^{\star}_{N-1}}\,\nabla\ve{u}_{N-1}\big\|^2 - \big\|\sqrt{\nu^{\star}_{0}}\,\nabla\ve{u}_{0}\big\|^2\right) \\
&+ (1-10\varepsilon)\sum_{n=2}^{N}\big\|\sqrt{\nu^{\star}_{n}}\,\nabla\ve{u}_{n}\big\|^2\, ,
\end{align*}
so that the previous estimate becomes
\begin{align*}
    & a_N + 4\tau(1-10\varepsilon)\sum_{n=2}^{N}\big\|\sqrt{\nu^{\star}_{n}}\,\nabla\ve{u}_{n}\big\|^2
      \\ &\leq C + 
    \frac{\tau\|\nabla\nu\|_{\infty}^2}{\varepsilon\nu_{\mathrm{min}}}\sum_{n=1}^{N-1}\left(\left\|\ve{u}_{n+1}^{\star}\right\|^2 + 32\varepsilon\tau\|\sqrt{\nu^{\star}_{n}}\,\nabla\ve{u}_{n}\|^2 + 8\varepsilon\tau\|\sqrt{\nu^{\star}_{n-1}}\,\nabla\ve{u}_{n-1}\|^2 \right)\\
    &\leq C + 
    \frac{\tau\|\nabla\nu\|_{\infty}^2}{\varepsilon\nu_{\mathrm{min}}}\sum_{n=1}^{N-1}\left(\left\|\ve{u}_{n+1}^{\star}\right\|^2 +\left\|\ve{u}_{n}\right\|^2 + 40\varepsilon\tau\|\sqrt{\nu^{\star}_{n}}\,\nabla\ve{u}_{n}\|^2 + 8\varepsilon\tau\|\sqrt{\nu^{\star}_{n-1}}\,\nabla\ve{u}_{n-1}\|^2 \right)\\
    &= C + 
    \frac{\tau\|\nabla\nu\|_{\infty}^2}{\varepsilon\nu_{\mathrm{min}}}\sum_{n=1}^{N-1} a_n \, ,
\end{align*}
where $a_N = \|\ve{u}_{N+1}^{\star}\|^2 + \|\ve{u}_{N}\|^2 + 40\varepsilon\tau\big\|\sqrt{\nu^{\star}_N}\,\nabla\ve{u}_N\big\|^2 + 8\varepsilon\tau\big\|\sqrt{\nu^{\star}_{N-1}}\,\nabla\ve{u}_{N-1}
    \big\|^2$ and $C = \|\ve{u}_{2}^{\star}\|^2 + \|\ve{u}_{1}\|^2  + 40\tau\varepsilon\big\|\sqrt{\nu^{\star}_{1}}\,\nabla\ve{u}_{1}\big\|^2 + 8\varepsilon\tau \big\|\sqrt{\nu^{\star}_{0}}\,\nabla\ve{u}_{0}\big\|^2$. Therefore, for $\varepsilon \leq 1/10$ the Grönwall Lemma \ref{Lem:Gronwall-unconditional} closes the proof.
}

\subsection{Generalized rotational formulation}
\begin{theorem}[Stability of the generalized rotational IMEX scheme]
For any time-step size $\tau=T/N$, $N\ge 2$, the IMEX scheme \eqref{BDF2ROT} satisfies the stability estimate
    \begin{equation}
    \begin{split}      
    &\big\|\ve{u}_{N}\big\|^2  + \big\|\ve{u}_{N+1}^{\star}\big\|^2 + 8\varepsilon\tilde{\gamma}\tau\nu_{\mathrm{min}}\left(5\big\|\nabla\ve{u}_{N}\big\|^2 + \big\|\nabla\ve{u}_{N-1}\big\|^2\right)  \\
    &+ 4(1-40\varepsilon)\tilde{\gamma}\tau\sum_{n=2}^{N}\nu_{\mathrm{min}}\big\|\nabla\ve{u}_{n}\big\|^2  \leq C\mathrm{exp}\left(\frac{\|\nabla\nu \|_{\infty}^2 T}{\varepsilon\tilde{\gamma}\nu_{\mathrm{min}}}\right)
     \, ,\label{stabilityROT}
     \end{split}
    \end{equation}
for any $\varepsilon \in \left(0,\frac{1}{40}\right]$, where $C = \|\ve{u}_{2}^{\star}\|^2 + \|\ve{u}_{1}\|^2  + 32\varepsilon\tilde{\gamma}\tau\nu_{\mathrm{\min}}(5\|\nabla\ve{u}_1\|^2 + \|\nabla\ve{u}_{0}\|^2)$ and
\begin{align*}
    \tilde{\gamma} := \mathrm{min}\lbrace\gamma,1\rbrace\, .
\end{align*}
\end{theorem}
\proof{Setting $\ve{v}=4\tau\ve{u}_{n+1}$ in \eqref{BDF2ROT} yields
\begin{align*}
 &\left(\|\ve{u}_{n+1}\|^2 + \|\ve{u}_{n+2}^{\star}\|^2\right) - \left(\|\ve{u}_{n}\|^2  + \|\ve{u}_{n+1}^{\star}\|^2\right) +\|\delta^2\ve{u}_{n+1}\|^2 + 4\tau\big\|\sqrt{\nu^{\star}_{n+1}}\,\nabla\times\ve{u}_{n+1}\big\|^2  \\ &+ 4\tau\gamma\big\|\sqrt{\nu^{\star}_{n+1}}\,\nabla\cdot\ve{u}_{n+1}\big\|^2 = 8\tau\langle 2\nabla^{\top}\ve{u}_{n}\nabla\nu_{n} -\nabla^{\top}\ve{u}_{n-1}\nabla\nu_{n-1},\ve{u}_{n+1}\rangle .
\end{align*}
For $\ve{u}_{n+1}\in\ve{H}^1_0(\Omega)$, we can write
\begin{align*}
\big\|\sqrt{\nu^{\star}_{n+1}}\,\nabla\times\ve{u}_{n+1}\big\|^2 + \gamma\big\|\sqrt{\nu^{\star}_{n+1}}\,\nabla\cdot\ve{u}_{n+1}\big\|^2  &\geq \nu_{\mathrm{min}}\left(\big\|\nabla\times\ve{u}_{n+1}\big\|^2 + \gamma\big\|\nabla\cdot\ve{u}_{n+1}\big\|^2\right)\\
&\geq \tilde{\gamma}\nu_{\mathrm{min}}\left(\big\|\nabla\times\ve{u}_{n+1}\big\|^2 + \big\|\nabla\cdot\ve{u}_{n+1}\big\|^2\right)\\
&= \tilde{\gamma}\nu_{\mathrm{min}}\big\|\nabla\ve{u}_{n+1}\big\|^2\, ,
\end{align*}
where $\tilde{\gamma} := \mathrm{min}\lbrace 1,\gamma \rbrace$. Hence:
\begin{align}
 &\left(\|\ve{u}_{n+1}\|^2 + \|\ve{u}_{n+2}^{\star}\|^2\right) - \left(\|\ve{u}_{n}\|^2  + \|\ve{u}_{n+1}^{\star}\|^2\right) +\|\delta^2\ve{u}_{n+1}\|^2 + 4\tau\tilde{\gamma}\nu_{\mathrm{min}}\big\|\nabla\ve{u}_{n+1}\big\|^2  \nonumber \\  &\leq  8\tau\langle 2\nabla^{\top}\ve{u}_{n}\nabla\nu_{n} -\nabla^{\top}\ve{u}_{n-1}\nabla\nu_{n-1},\ve{u}_{n+1}\rangle \, .\label{estROT1}
\end{align}
The right-hand side of \eqref{estROT1} is estimated similarly as in the Laplacian case:
\begin{align*}
    &8\tau\langle 2\nabla^{\top}\ve{u}_{n}\nabla\nu_{n} -\nabla^{\top}\ve{u}_{n-1}\nabla\nu_{n-1},\ve{u}_{n+1}\rangle \leq \left\|\delta^2\ve{u}_{n+1}\right\|^2 + 32\varepsilon\tilde{\gamma}\tau\nu_{\mathrm{min}}(4\|\nabla\ve{u}_{n}\|^2 + \|\nabla\ve{u}_{n-1}\|^2) \\
    & + \frac{\tau\|\nabla\nu\|_{\infty}^2}{\varepsilon\tilde{\gamma}\nu_{\mathrm{min}}}\left(128\varepsilon\tilde{\gamma}\tau\nu_{\mathrm{min}}\|\nabla\ve{u}_{n}\|^2 + 32\varepsilon\tilde{\gamma}\tau\nu_{\mathrm{min}}\|\nabla\ve{u}_{n-1}\|^2 + \left\|\ve{u}_{n+1}^{\star}\right\|^2\right).
\end{align*}
Hence:
\begin{align*}
 &\|\ve{u}_{n+1}\|^2 + \|\ve{u}_{n+2}^{\star}\|^2  + 4\tau\tilde{\gamma}\nu_{\mathrm{min}}\Big(\big\|\nabla\ve{u}_{n+1}\big\|^2 - 128\varepsilon\|\nabla\ve{u}_{n}\|^2 - 32\varepsilon\|\nabla\ve{u}_{n-1}\|^2\Big) \leq \nonumber \\  & \|\ve{u}_{n}\|^2  + \|\ve{u}_{n+1}^{\star}\|^2  +\frac{\tau\|\nabla\nu\|_{\infty}^2}{\varepsilon\tilde{\gamma}\nu_{\mathrm{min}}}\left(128\varepsilon\tilde{\gamma}\tau\nu_{\mathrm{min}}\|\nabla\ve{u}_{n}\|^2 + 32\varepsilon\tilde{\gamma}\tau\nu_{\mathrm{min}}\|\nabla\ve{u}_{n-1}\|^2 + \left\|\ve{u}_{n+1}^{\star}\right\|^2\right) \, .\label{estROT2}
\end{align*}
We then add up from $n=1$ to $n=N-1$ and use \eqref{equality}:
\begin{align*}
    &\sum_{n=1}^{N-1}\left(\big\|\nabla\ve{u}_{n+1}\big\|^2 - 32\varepsilon\big\|\nabla\ve{u}_{n}\big\|^2 - 8\varepsilon\big\|\nabla\ve{u}_{n-1}\big\|^2\right)\leq \\
&40\varepsilon\left(\big\|\nabla\ve{u}_{N}\big\|^2 - \big\|\nabla\ve{u}_{1}\big\|^2\right) + 8\varepsilon\left(\big\|\nabla\ve{u}_{N-1}\big\|^2 - \big\|\nabla\ve{u}_{0}\big\|^2\right) + (1-40\varepsilon)\sum_{n=2}^{N}\big\|\nabla\ve{u}_{n}\big\|^2\, ,
\end{align*}
so as to get
\begin{align*}
    & a_N + 4(1-40\varepsilon)\tau\tilde{\gamma}\sum_{n=2}^{N}\nu_{\mathrm{min}}\big\|\nabla\ve{u}_{n}\big\|^2
      \\ &\leq C + 
    \frac{\tau\|\nabla\nu\|_{\infty}^2}{\varepsilon\tilde{\gamma}\nu_{\mathrm{min}}}\sum_{n=1}^{N-1}\left(\left\|\ve{u}_{n+1}^{\star}\right\|^2 + 128\varepsilon\tau\|\nabla\ve{u}_{n}\|^2 + 32\varepsilon\tau\|\nabla\ve{u}_{n-1}\|^2 \right)\\
    &\leq  C + 
    \frac{\tau\|\nabla\nu\|_{\infty}^2}{\varepsilon\tilde{\gamma}\nu_{\mathrm{min}}}\sum_{n=1}^{N-1} a_n \, ,
\end{align*}
where $a_N = \|\ve{u}_{N+1}^{\star}\|^2 + \|\ve{u}_{N}\|^2 + 32\varepsilon\tilde{\gamma}\tau\nu_{\mathrm{\min}}(5\|\nabla\ve{u}_N\|^2 + \|\nabla\ve{u}_{N-1}
    \|^2)$ and $C = \|\ve{u}_{2}^{\star}\|^2 + \|\ve{u}_{1}\|^2  + 32\varepsilon\tilde{\gamma}\tau\nu_{\mathrm{\min}}(5\|\nabla\ve{u}_1\|^2 + \|\nabla\ve{u}_{0}\|^2)$. The proof is completed by using the Grönwall inequality for $\varepsilon \leq 1/40$.
}

\subsection{Comparison}
The analyses presented above show that the three alternative formulations, all of which are mathematically consistent, yield stability estimates without imposing a time-step restriction associated with the IMEX treatments. An important distinction is that the stress-divergence formulation satisfies a stronger estimate: it yields kinetic energy decay, whereas the generalized Laplacian and rotational formulations lead to estimates with possible growth factors. This growth has not been observed in the numerical examples reported below or in previous applications of the generalized Laplacian formulation, even for large $\tau$ \cite{Barrenechea2024}. 

From a computational standpoint, the generalized Laplacian method offers the advantage of having a much sparser velocity matrix than in the other two variants. All three IMEX methods circumvent the need for fixed-point iterations and are still formally second-order accurate in time. Table \ref{tableComparison} summarizes the key differences between all three formulations. 
\begin{table}[ht!]
 \centering
 \caption{Comparison between the theoretical and numerical properties of the (IMEX) stress-divergence (SD), generalized Laplacian (GL) and rotational (ROT) methods.}
 {\begin{tabular}{|c|c|c|c|c|}
    \hline
   \textbf{Form}   & \textbf{Assumption} & \textbf{Stabilization} & \textbf{Growth factor}  & \textbf{Natural BC}\\
   \hline
   SD   & $\nu_n\in L^{\infty}(\Omega)$ & not needed & none & $(2\nu\nabla^{\mathrm{s}}\ve{u})\ve{n}-p\ve{n}$ \\ \hline   
   
   GL  & $\nu_n\in W^{1,\infty}(\Omega)$ & not needed & $\exp\big(\frac{10\|\nabla\nu\|_{\infty}^2T}{\nu_{\mathrm{min}}}\big)$ & $(\nu\nabla\ve{u})\ve{n}-p\ve{n}$\\ \hline   
   
   ROT  & $\nu_n\in W^{1,\infty}(\Omega)$ & grad-div, $\gamma >0$ & $\exp\big(\frac{40\|\nabla\nu\|_{\infty}^2T}{\gamma\nu_{\mathrm{min}}}\big)$ & $(\nu\nabla\times\ve{u})\times\ve{n}-p\ve{n}$ \\
    \hline
 \end{tabular}}
  \label{tableComparison} 
\end{table}

\section{Numerical examples}\label{sec5}
In this section, we consider a set of two-dimensional flow examples to compare the three viscous formulations in terms of temporal accuracy, outflow-boundary behavior, and qualitative flow features. All tests are performed with $\ve{f}=\ve{0}$.
The numerical experiments were implemented using the open-source finite element library Firedrake \cite{FiredrakeUserManual}, and the resulting linear systems were solved with a direct LU factorization using MUMPS \cite{MUMPS:1,MUMPS:2}. The spatial and temporal discretizations are those described in Section \ref{sec3}.

\subsection{Temporal convergence test}
\label{sec:temporal_convergence}
To assess temporal convergence for the different variable-viscosity formulations, we consider an exact unsteady solution of the Navier--Stokes equations. The viscosity field is
\begin{equation}
 \nu(x,y,t)=xy\,f(t)+g(t),   
\end{equation}  where $f$ and $g$ are chosen to that $\nu > 0$. A corresponding exact solution is
\begin{equation}
\label{eq:manufactured_fields}
\ve{u}(x,y,t)= 2f(t)
\begin{pmatrix}
y\\
x
\end{pmatrix} \qquad \text{ and }
\qquad
p(x,y,t)=\left(C-2xy\right)f'(t)
\end{equation} 
for an arbitrary $C$. For this test, we choose $\Omega = (0,1)^2$, $T=3$, $f(t)=\sin^2(t)$ and $g(t)=0.001$, and the initial and Dirichlet data are computed from the analytical solution. The constant $C$ is set so that $p \in L_0^2(\Omega)$ for all $t$, that is,

\[0= 
\int_\Omega p(x,y,t)\,\mathrm{d}\Omega
=
f'(t)\int_0^1\!\!\int_0^1 \left(C-2xy\right)\,\mathrm{d}x\,\mathrm{d}y\, ,
\]
which gives us $C=1/2$. This normalization is required because the test uses $\Gamma_N=\emptyset$, so the pressure at each time is determined only up to an additive constant.

The domain $\Omega$ is discretized in a uniform quadrilateral mesh of \(4\times 4\) square Taylor--Hood elements. Since the exact velocity is affine in space and the exact pressure is bilinear, both fields belong, at each time, to the corresponding finite element spaces. Therefore, the spatial error is negligible up to quadrature and solver tolerances, so the benchmark isolates the temporal discretization error. To avoid polluting the temporal study with start-up errors, we initialize the first two time steps from the exact solution \eqref{eq:manufactured_fields}: $\ve{u}_0=\ve{u}(\ve{x},0)$, $\ve{u}_1=\ve{u}(\ve{x},\tau)$.

The temporal convergence study is performed by evaluating the numerical solution at
the fixed time $t=3$,  for the sequence of time-step sizes
\begin{equation}
\label{eq:tau_sequence}
\tau\in\left\{1,\;10^{-1},\;10^{-2},\;10^{-3}\right\}.
\end{equation}
Since the spatial discretization is kept fixed and the manufactured fields are
represented exactly in space, the refinement in \eqref{eq:tau_sequence} isolates the
dependence of the error on \(\tau\). At the final evaluation time \(t=3\), the numerical velocity and pressure are compared with the exact solution \eqref{eq:manufactured_fields}. The reported errors are
\begin{equation}
\label{eq:error_measures}
e_u(\tau)
=
\|\ve{u}(\ve{x},T)
-
\ve{u}_h(\ve{x},T)\|_{L^2(\Omega)},
\qquad
e_p(\tau)
= 
\|p(\ve{x},T)
-
p_h(\ve{x},T)\|_{L^2(\Omega)}.
\end{equation}
Table~\ref{tab:temporal_convergence} summarizes the results, which confirm the expected second order for all three formulations. Among the cases considered, the generalized Laplacian formulation gives the smallest errors in this test, although all three formulations exhibit the expected second-order temporal convergence.  


\begin{table}[h!]
\centering
\caption{Temporal convergence history for the exact-solution test at
\(t=3\).}
\label{tab:temporal_convergence}
\begin{tabular}{|c|c|c|c|c|c|}
\hline
Method & $\tau$ & $e_u(\tau)$ & rate & $e_p(\tau)$ & rate \\
\hline
\multirow{4}{*}{ROT}
& $1$       & $8.253\times10^{-1}$ & -     & $4.595\times10^{-1}$ & -     \\
& $10^{-1}$ & $7.758\times10^{-3}$ & 2.027 & $4.371\times10^{-3}$ & 2.022 \\
& $10^{-2}$ & $6.007\times10^{-5}$ & 2.111 & $2.559\times10^{-5}$ & 2.233 \\
& $10^{-3}$ & $5.793\times10^{-7}$ & 2.016 & $2.416\times10^{-7}$ & 2.025 \\
\hline
\multirow{4}{*}{SD}
& $1$       & $3.313\times10^{-1}$ & -     & $1.274\times10^{-1}$ & -     \\
& $10^{-1}$ & $4.941\times10^{-3}$ & 1.826 & $8.438\times10^{-4}$ & 2.179 \\
& $10^{-2}$ & $4.919\times10^{-5}$ & 2.002 & $5.257\times10^{-6}$ & 2.206 \\
& $10^{-3}$ & $4.910\times10^{-7}$ & 2.001 & $5.029\times10^{-8}$ & 2.019 \\
\hline
\multirow{4}{*}{GL}
& $1$       & $8.713\times10^{-1}$ & -     & $2.463\times10^{-1}$ & -     \\
& $10^{-1}$ & $5.095\times10^{-3}$ & 2.233 & $1.238\times10^{-3}$ & 2.299 \\
& $10^{-2}$ & $4.040\times10^{-5}$ & 2.101 & $3.574\times10^{-6}$ & 2.539 \\
& $10^{-3}$ & $3.915\times10^{-7}$ & 2.014 & $2.917\times10^{-8}$ & 2.088 \\
\hline
\end{tabular}
\end{table}


\subsection{Venturi channel}
The second example considers flow through a two-dimensional Venturi channel. The domain has length $L=5$, inlet and outlet height $H=1$, and a smooth symmetric contraction. The throat is located at $x=0$, where
the channel height is $H_{\mathrm{th}}=0.4$. In the contraction region, the half-height is prescribed as
$h(x)=0.2+0.15[1-\cos(2\pi x/3)]$ for $x\in[-1.5,1.5]$, while $h(x)=0.5$ along the rest of the domain, $1.5 \leq |x| \leq 2.5$.

The fluid is described by the Carreau--Yasuda law \eqref{CY}, with
\begin{equation*}
\nu_{\mathrm{max}}=10,\qquad
\nu_{\mathrm{min}}=1,\qquad
\lambda=4,\qquad
a=1.25,\qquad
n=0.25.
\end{equation*}
A parabolic velocity profile with maximum value $u_{\max}=15$ is prescribed at the inlet, no-slip Dirichlet conditions are imposed
on the walls, and homogeneous natural boundary conditions are imposed at the outlet. The computation is performed up to $T=3$, using $\tau=10^{-2}$ and 6356 Taylor--Hood
$(\mathbb{P}_2/\mathbb{P}_1)$ elements.

We first compare shear-stress contours. The Cauchy stress tensor is given by
\begin{equation}
\sigma=-p\mathbb{I}+2\nu\nabla^{\mathrm{s}}\ve{u}\, ,
\end{equation}
and we report its off-diagonal component $\sigma_{xy}$. Figure~\ref{fig:venturi_isoclines} shows the results for the GL, SD, and ROT formulations. The downstream region reflects the different outflow treatments associated with each formulation. Consistently with what is known for constant-viscosity flows, the GL formulation with homogeneous pseudo-traction conditions provides the closest approximation to a developed outflow, as indicated by nearly parallel contours near the outlet.

\begin{figure}[h!]
    \centering

    \subfigure[GL]{%
        \includegraphics[width=0.85\textwidth]{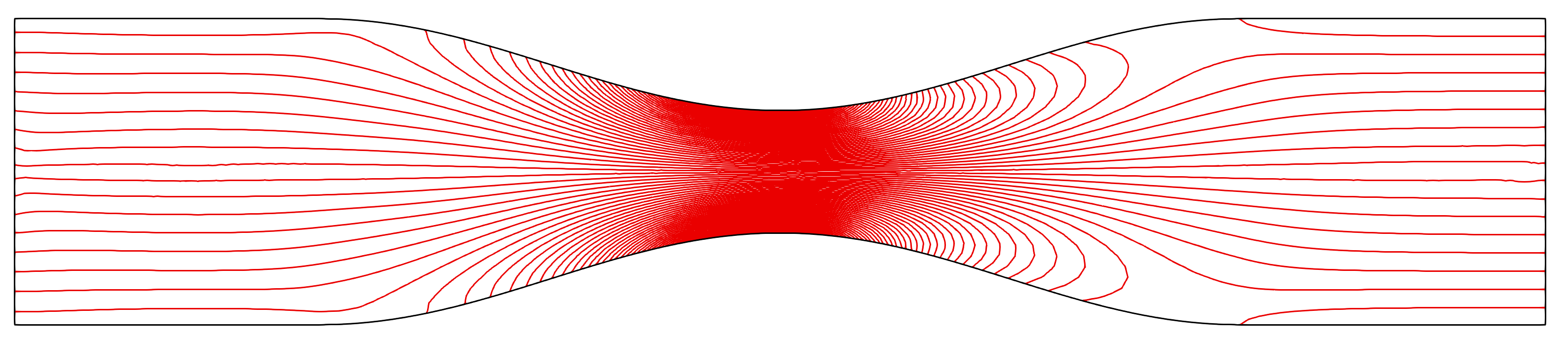}%
        \label{fig:venturi_isoclines_gl}%
    }
    \subfigure[SD]{%
        \includegraphics[width=0.85\textwidth]{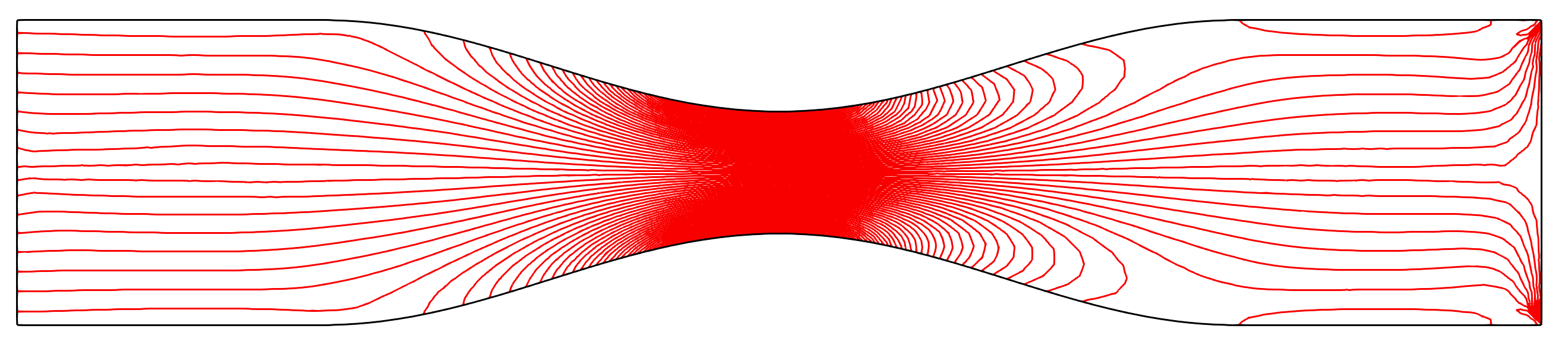}%
        \label{fig:venturi_isoclines_sd}%
    }
    \subfigure[ROT]{%
        \includegraphics[width=0.85\textwidth]{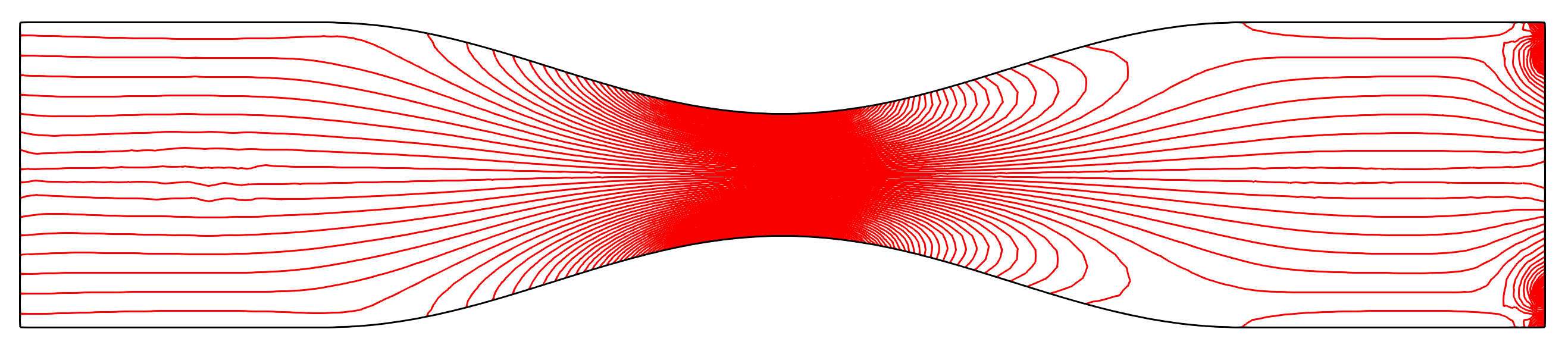}%
        \label{fig:venturi_isoclines_rot}%
    }

    \caption{Venturi channel. Isolines of the shear-stress component $\sigma_{xy}$ for the GL, SD, and ROT formulations with the corresponding homogeneous natural outflow conditions.}
    \label{fig:venturi_isoclines}
\end{figure}

Next, we assess the influence of the outlet location on the pressure field for each formulation by considering truncated domains. Pressure profiles are sampled along the centerline using three computational domains: the complete Venturi tube used above, a truncated domain retaining
$80\%$ of the original length, and another one retaining $50\%$ of the
original length. The truncation procedure is illustrated in Figure~\ref{fig:venturi_truncation_marks}.

\begin{figure}[h!]
    \centering
    \includegraphics[width=0.65\textwidth]{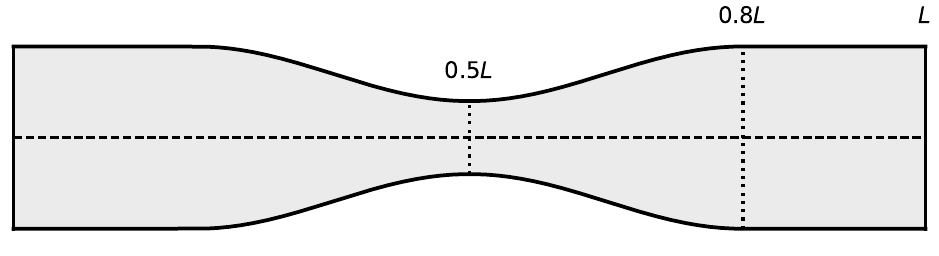}
    \caption{Venturi channel. Complete reference domain, centerline used for pressure sampling, and outlet locations used for the truncated-domain pressure profiles.}
    \label{fig:venturi_truncation_marks}
\end{figure}

The corresponding pressure profiles are shown in Figure~\ref{fig:venturi_pressure} for the GL, SD, and ROT formulations. The GL pseudo-traction formulation is less sensitive to domain truncation, with pressure curves that remain nearly overlapping even for the severe $50\%$ truncation.

\begin{figure}[H]
    \centering

    \subfigure[GL]{%
        \includegraphics[width=0.92\textwidth]{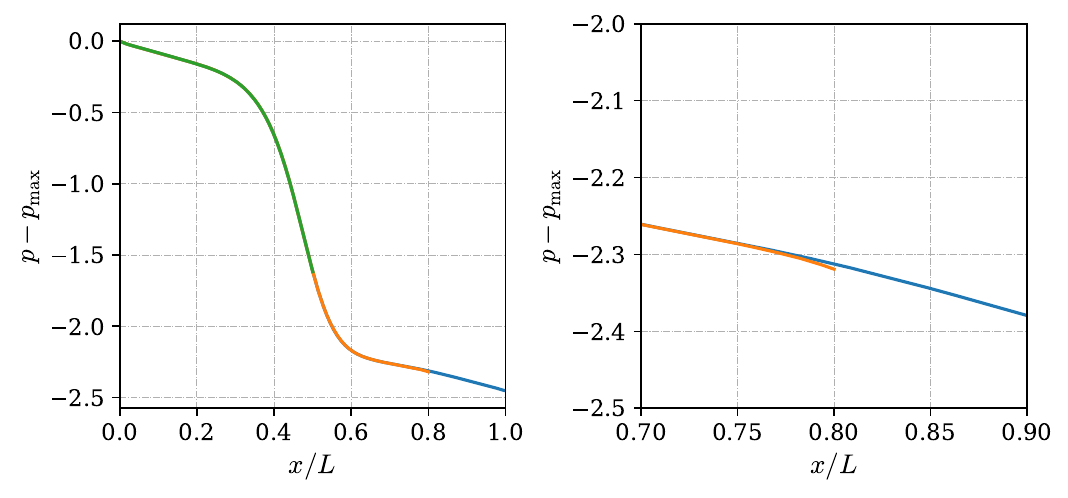}%
        \label{fig:venturi_pressure_gl}%
    }

    \subfigure[SD]{%
        \includegraphics[width=0.92\textwidth]{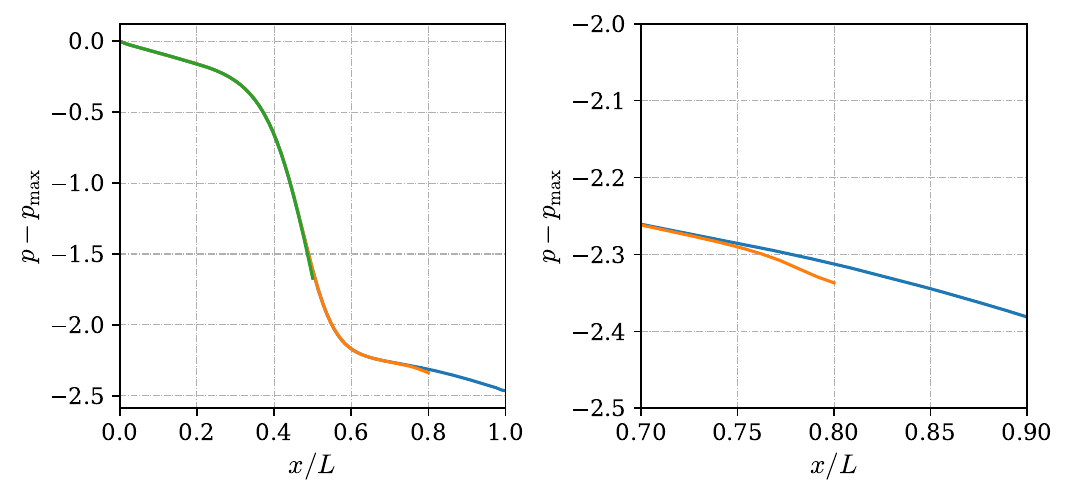}%
        \label{fig:venturi_pressure_sd}%
    }

    \subfigure[ROT]{%
        \includegraphics[width=0.92\textwidth]{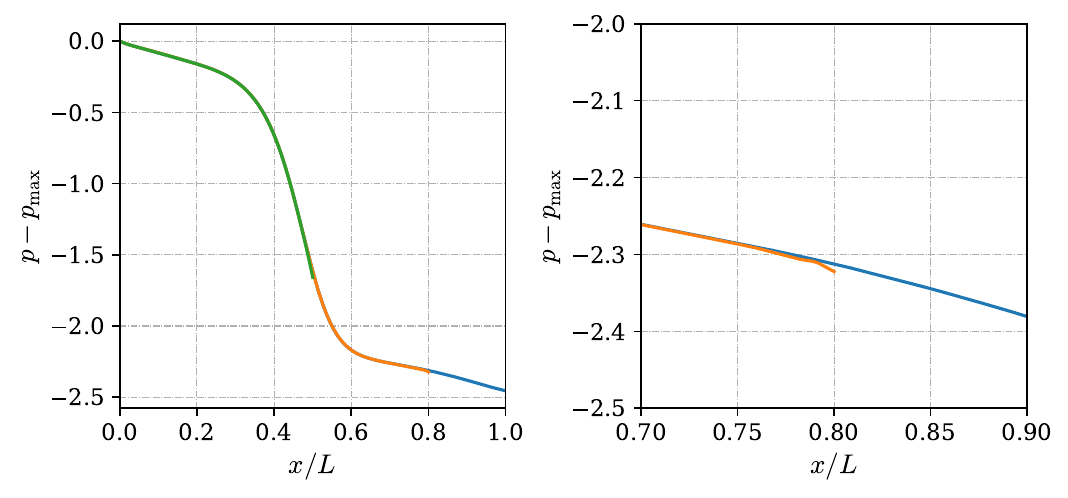}%
        \label{fig:venturi_pressure_rot}%
    }

    \caption{Venturi channel. Centerline pressure profiles computed on the complete domain, on the $80\%$-length truncated domain, and on the $50\%$-length truncated domain.}
    \label{fig:venturi_pressure}
\end{figure}

\subsection{Von Kármán vortex street}

The last example considers flow past a cylinder in a two-dimensional channel, leading to the formation of a von Kármán vortex street. The fluid is again described by the Carreau--Yasuda law \eqref{CY}, with the same rheological parameters considered in the previous example, but now with a parabolic inlet profile with maximum
velocity \(u_{\max}=75\), no-slip on the channel walls and on the cylinder boundary, and homogeneous natural boundary on the outlet. The simulation is performed up to \(T=3\), using \(\tau=10^{-2}\) and 17814 elements.

This example is used to evaluate whether the different natural outflow conditions
affect the transport of vortical structures through the outlet. The vortex street is
expected to be convected downstream and to leave the computational domain in a
similar manner for the three formulations. The corresponding visual comparison is shown in Figure~\ref{fig:vk_vortices}. In this case, the three formulations produce very similar streamline patterns. This behavior is consistent with the downstream flow being dominated by vortex transport, for which convective effects are more prominent. Together with the previous example, these results suggest that the differences induced by the outflow treatment are more pronounced in lower-Reynolds regimes. For more convection-dominated flows, other aspects, such as matrix sparsity, may become more relevant.

\begin{figure}[h!]
\centering
\subfigure[GL]{%
    \includegraphics[width=0.99\textwidth]{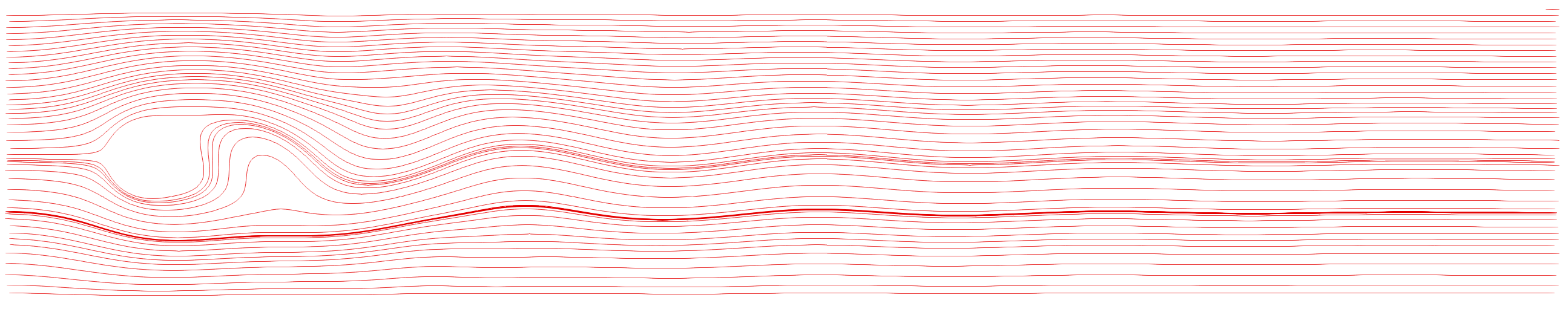}%
    \label{fig:vk_vortices_gl}%
}

\subfigure[SD]{%
    \includegraphics[width=0.95\textwidth]{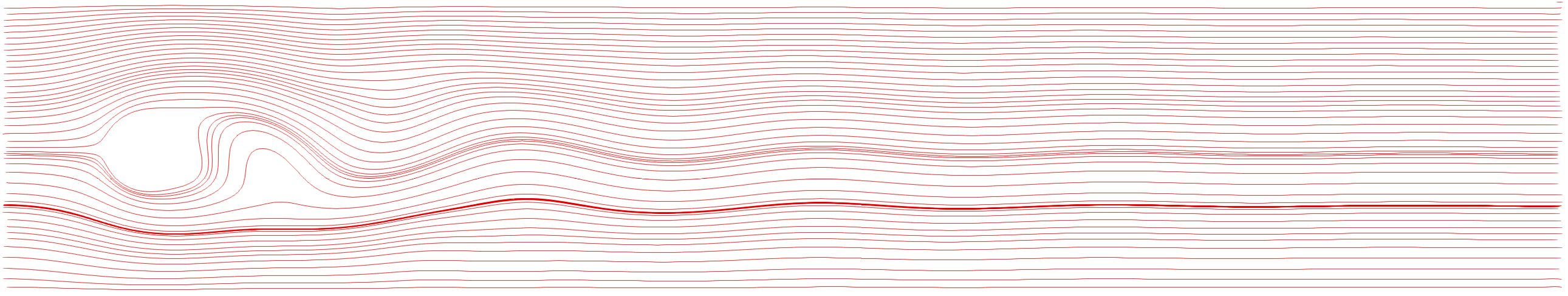}%
    \label{fig:vk_vortices_sd}%
}

\subfigure[ROT]{%
    \includegraphics[width=0.95\textwidth]{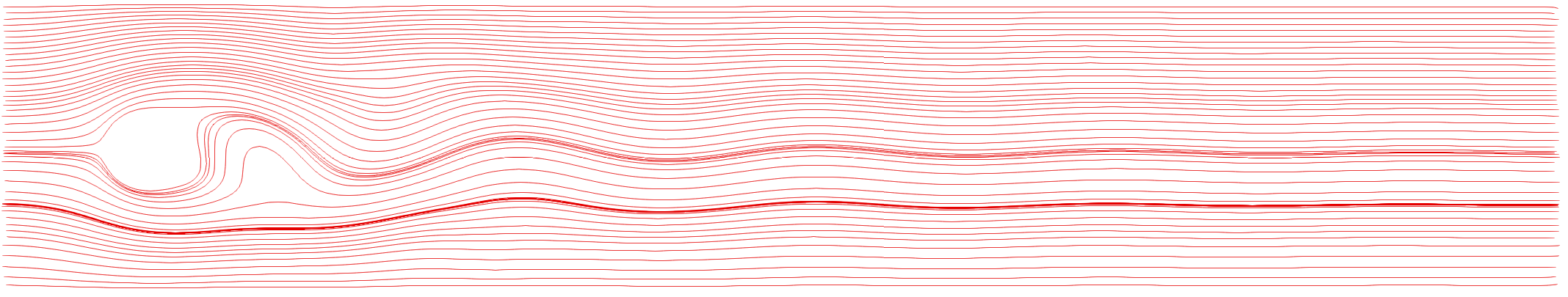}%
    \label{fig:vk_vortices_rot}%
}

\caption{Von Kármán vortex street. Streamlines for the GL, SD, and ROT formulations at \(t=3\).}
\label{fig:vk_vortices}
\end{figure}

\newpage
\section{Concluding remarks}\label{sec6}
In this work, we have compared the classical stress-divergence (SD) form of the Navier--Stokes system with the variable-viscosity generalization of the Laplacian (GL) and rotational (ROT) forms. These generalized formulations recover the natural outflow boundary conditions of their constant-viscosity counterparts. The formulations differ in four important aspects:
\begin{enumerate}
    \item \textsl{Outflow BCs:} the SD variant is formulated in terms of normal tractions, which may be more appropriate, e.g., in fluid-structure simulations; by contrast, the pseudo-tractions induced by the GL formulation appear better suited for truncated domains, as suggested by the numerical examples considered here.
    \item \textsl{Matrix sparsity:} with an appropriate IMEX treatment of the momentum equation, all formulations lead to symmetric diffusion matrices; however, the corresponding matrix in the IMEX-GL case is sparser and has a block-diagonal structure with $d$ identical velocity blocks.
    \item \textsl{Spatial discretization:} the SD and GL formulations allow for standard Lagrangian interpolation in space, while the ROT variant requires either $H($curl,$\Omega)$ elements or grad-div stabilization.
    \item \textsl{Temporal stability:} according to our theoretical analysis, the GL and ROT methods might induce a non-physical energy growth for long-term simulations; however, we have not observed such phenomena in numerical experiments, which indicates that our analysis may not be sharp. 
    \end{enumerate}

    Rather than advocate one formulation over the others, our purpose is to provide computational rheologists with information to assist them in selecting the best-suited method for the applications at hand. In the flow examples considered here, the ROT variant did not show a clear practical advantage over the SD or GL formulations. Nevertheless, the rotational formulation presented and analyzed here may be useful in other settings involving curl-based operators, such as wave propagation in non-homogeneous media \cite{Boffi2026,Srmny2013}. This direction will be explored in future work.

\section*{Acknowledgments}
DRQP acknowledges funding by the Federal Ministry of Education and Research (BMBF) and the Ministry of Culture and Science of the German State of North Rhine-Westphalia (MKW) under the Excellence Strategy of the Federal Government and the Länder. E. Castillo and F. Galarce gratefully acknowledge the financial support provided by ANID Chile through FONDECYT Regular Project No. 1250287.

\bibliographystyle{unsrtnat}
\bibliography{references}%

\end{document}